\documentclass[11pt]{amsart}
\usepackage{amscd}
\usepackage[all]{xy}
\usepackage{graphicx}
\usepackage{amsmath}
\usepackage{amsfonts}
\usepackage{amssymb}
\usepackage{latexsym}
\usepackage{slashed}
\usepackage{soul}
\usepackage{comment}
\usepackage{color}
\usepackage{rotating}

\definecolor{brown}{RGB}{150,100,0}

\definecolor{purple}{RGB}{128,0,128}

\definecolor{grey}{RGB}{128,128,128}

\usepackage{amsmath, latexsym, amssymb}
\numberwithin{equation}{section}
\theoremstyle{plain}
\newtheorem{lemma}{Lemma}[section]
\newtheorem{proposition}[lemma]{Proposition}
\newtheorem{theorem}[lemma]{Theorem}
\newtheorem{corollary}[lemma]{Corollary}

\theoremstyle{definition}

\newtheorem{remark}[lemma]{Remark}
\newtheorem{example}[lemma]{Example}

\newcommand{\R}{{\mathbb R}}
\newcommand{\C}{{\mathbb C}}
\newcommand{\F}{{\mathbb F}}

\newcommand{\Z}{{\mathbb Z}}

\newcommand{\Cc}{{\mathcal C}}    

\newcommand{\Rr}{{\mathcal R}}

\newcommand{\Ad}{{\rm Ad \,}}
\newcommand{\Om}{{\Omega}}
\newcommand{\om}{{\omega}}
\newcommand{\eps}{{\varepsilon}}
\newcommand{\Di}{{\rm Diff}}

\newcommand{\e}{{\mathfrak e}}

\newcommand{\g}{{\mathfrak g}}
\newcommand{\gl}{{\mathfrak {gl}}}
\newcommand{\h}{{\mathfrak h}}

\renewcommand{\u}{{\mathfrak u}}
\newcommand{\su}{{\mathfrak su}}

\newcommand{\ssl}{{\mathfrak {sl}}}

\newcommand{\la}{\langle}
\newcommand{\ra}{\rangle}

\newcommand{\INTO}{\hookrightarrow}

\newcommand{\rk}{{\rm rk }}

\newcommand{\St}{{\rm St}}
\newcommand{\Gal}{{\rm Gal}}
\newcommand{\GL}{{\rm GL}}
\newcommand{\SL}{{\rm SL}}
\newcommand{\Aut}{{\rm Aut}}
\newcommand{\SO}{{\rm SO}}
\newcommand{\U}{{\rm U}}

\date{\today}
\title[Classification of  $k$-forms  and   associated  geometry]
{Classification of  $k$-forms  on  $\R^n$   and   the existence  of  associated  geometry  on  manifolds}

\author{H\^ong V\^an L\^e}
\address{Institute of Mathematics  of the Czech  Academy of Sciences,
Zitna 25, 11567 Praha 1, Czech Republic}
\email{hvle@math.cas.cz}

\author{Ji\v r\'i  Van\v zura}
\address{Institute of Mathematics  of the Czech Academy of Sciences,
Zitna 25, 11567 Praha 1, Czech Republic}
\email{vanzura@math.cas.cz}

\thanks{The research of  HVL was  supported by the GA\v CR-project 18-00496S and  RVO:67985840.}

\begin{document}
	
\abstract  In this paper  we survey   methods  and results  of  classification  of $k$-forms (resp. $k$-vectors  on $\R^n$), understood as    description of the orbit space of the standard  $\GL(n, \R)$-action  on $\Lambda^k \R^{n*}$ (resp. on $\Lambda ^k \R^n$). We   discuss  the existence  of  related  geometry defined
by differential forms  on   smooth manifolds.  This paper also contains an Appendix by Mikhail Borovoi  on   Galois cohomology methods for  finding real forms  of complex orbits. 
\endabstract


\maketitle

\section*{Preface}

Hamiltonian systems   were  one of  research  topics of  H\^ong V\^an L\^e in her undergraduate  study  and  calibrated geometry  was    the topic  of  her Ph.D. Thesis  under   guidance   of      Professor Anatoly   Timofeevich Fomenko.
Hamiltonian  systems  are defined  on symplectic  manifolds  and  calibrated   geometry
is  defined  by  closed  differential forms of  comass  one on  Riemannian manifolds.  Since that time  she  works frequently   on geometry     defined  by  differential  forms,   some  of her papers  were  written in collaboration with   Ji\v ri Van\v zura, \cite{LPV2008, LV2015, LV2017}.    We   dedicate  this survey on algebra and geometry of $k$-forms on $\R^n$  as well as on  smooth manifolds to Anatoly Timofeevich Fomenko  on the occasion of his 75th birthday and we wish him good
health, happiness and much success for the coming years.

\section{Introduction}\label{sec:intr}
Differential forms   are  excellent   tools for  the study of   geometry  and topology  of  manifolds  and their submanifolds as well as  dynamical systems   on them. 
 K\"ahler  manifolds, and more generally,    Riemannian manifolds  $(M, g)$  with non-trivial holonomy group  admit   parallel    differential forms and hence   calibrations   on $(M, g)$ \cite{HL1982},
\cite{Salamon1989}, \cite{LV2017}, \cite{FLSV2018}.  In the study  of     Riemannian manifolds  with  non-trivial holonomy groups   these  parallel  differential forms  are
 extremely important \cite{Bryant1987}, \cite{Joyce2000}. In  their seminal paper \cite{HL1982} Harvey-Lawson   used  calibrations  as  powerful tool  for the study of  geometry of  calibrated  submanifolds,   which   are volume minimizing.  Their paper  opened  a new field of calibrated  geometry   \cite{Joyce2007}  where   one finds more and  more  tools
 for the study  of calibrated   submanifolds using  differential forms, see e.g.,  \cite{FLSV2018}.  In 2000 Hitchin  initiated   the study of geometry defined by  a differential 3-form \cite{Hitchin2000},   and  in  a subsequent paper he   analyzed beautiful  geometry defined by differential forms in  low dimensions \cite{Hitchin2001}.  One starts  investigation   of  a differential  form  $\varphi^k$  of degree $k$  on a  manifold $M^n$ of dimension  $n$ by finding a   {\it normal  form}   of $\varphi^k$    at  a  point $x\in M^n$   and, if possible,  to find a normal  form of $\varphi^k$ up to  certain order  in a small neighborhood   $U(x) \subset   M^n$.  Finding a normal form  of   $\varphi^k$  at a point  $x \in M^n$  is the  same as     finding 
 a  canonical representative of  the equivalence class  of  $\varphi^k(x)$ in $\Lambda^k(T^*_xM^n)$, where   two $k$-forms
 on $T_xM^n$ are {\it equivalent} if they are in the same orbit  of the  standard $\GL(n,\R)$-action on $\Lambda ^k (T_x^*M^n)= \Lambda^k\R^{n*}$. We say that a manifold  $M^n$  is endowed  by  a differential  form $\varphi \in \Om^*(M^n)$ of   type $\varphi_0 \in \Lambda^* \R^{n*}$, if   for  all $ x\in M^n$  the equivalence class  of  $\varphi(x)\in \Lambda^* T^*_x M^n$  can be identified  with   the equivalent class  of  $\varphi_0 \in \Lambda^* \R^{n*}$ 
  via  a linear  isomorphism  $T_x M^n = \R^n$.
 Instead of   investigation of a   normal form of  a concrete    form $\varphi^k$,  we  may be  also interested
   in  a classification  of (equivalent) $k$-forms on $\R^n$,  understood  as   a description of the  moduli  space  of  equivalent $k$-forms on $\R^n$,  which could give us insight   on  a normal form of $\varphi^k$ and  could also  suggest   interesting  candidates  for the  geometry  defined by differential forms. Classification
 of $k$-forms on $\R^n$  is   a part  of  algebraic invariant  theory.
Recall  that {\it an  invariant} of an   equivalence relation on  a set  $S$, e.g., defined by   orbits of  an action   of a group $G$ on $S$,
is a mapping   from $S$ to  another   set $Q$  that is constant  on the equivalence classes.
A  system of invariants   is called {\it complete} if  it separates  any two equivalent classes.
If  a complete  system of invariants   consists of  one  element, we  call this invariant  complete.
In the  classical algebraic invariant  theory  one deals  mainly  with   actions  of classical or algebraic groups  on  some space of tensors of  a fixed type over a  vector space over  a field $\F$  \cite{Gurevich1964},   see   \cite{PV1994} for  a survey of modern  invariant theory  and source of  algebraic invariant theory.
From a geometric point  of view,  the most important invariants  of  a  form $\varphi^k$ on $\R^n$ are {\it the rank  of
$\varphi^k$} and {\it the stabilizer}    of  $\varphi^k$  under the action of  $\GL(n, \R)$. Recall that   the rank  of $\varphi^k$, denoted by $\rk \, \varphi^k$,   is  the dimension  of the  image  of the  linear operator $L_{\varphi^k}: \R^n \to \Lambda ^{k-1}  \R^{n*},  \:  v\mapsto  i_v \varphi^k$.  We denote  the stabilizer  of $\varphi^k$  by  $\St_{\GL(n, \R)} (\varphi^k)$,  and  in general,  we denote  by  $\St_G(x)$ the stabilizer  of  a point $x$ in a set $S$ where a group $G$  acts.
 A  form $\varphi^k \in \Lambda ^k \R^{n*}$  is called {\it non-degenerate}, or {\it multisymplectic}, if  $\rk \,  \varphi^k = n$.  Furthermore, it is  important   to study the topology of the
orbit $\GL(n, \R)\cdot \varphi ^k= \GL(n, \R)/ \St_{\GL(n, \R)}(\varphi^k)$,  for example,   the connectedness, see  Proposition \ref{prop:connected} below, the openness, the closure  of the  orbit $\GL(n, \R)\cdot \varphi^k \subset \Lambda^k \R^{n*}$.  It turns  out that
understanding these questions   helps us to understand  the structure  of the     orbit space of $\GL(n,\R)$-action  on $\Lambda ^k \R^{n*}$.  These invariants  of $k$-forms
 shall be highlighted  in our survey.

Let us outline  the plan  of our paper.
In the first  part of  Section  2 we make  several observations on  the  duality between   $\GL(n, \R)$-orbits  of $k$-forms  on $\R^n$ and  $\GL(n,\R)$-orbits of $k$-vectors  as well as
the   duality  between $\GL^+(n, \R)$-orbits  of $k$-forms on $\R^n$ and  $\GL^+(n,\R)$-orbits
of $(n-k)$-forms on $\R^n$. Then we recall the classification   of 2-forms  on $\R^n$ (Theorem \ref{thm:k=2}) and   present the Martinet's classification of $(n-2)$-forms on $\R^n$ (Theorem \ref{thm:k=n-2}).

In contrast to the   classification  of 2-forms on $\R^n$,  the classification  of 3-forms  on $\R^n$  depends
on the  dimension $n$. Since  $\dim \Lambda ^3  \R^{n*} \ge \dim \GL(n, \R) +1$, if $n \ge 9$,
there are infinite numbers of inequivalent  3-forms  in $\R^n$.
Till now  there is  no  classification   of the  $\GL(n, \R)$-action on $\Lambda^3 \R^{n*}$, if  $n \ge 10$.

In the  dimension   $n=9$
the  classification  of  the $\SL(9, \C)$-orbits on $\Lambda ^3 \C^9$  has been obtained by  Vinberg-Elashvili \cite{VE1978}. In the second part  of Section 2  we    survey  Vinberg-Elashvili's  result  and   some  further developments  by  
Le \cite{Le2011} and Dietrich-Facin-de  Graaf
\cite{DFG2015},
  which  give partial information  on $\GL(9,\R)$-orbits on $\Lambda^3\R^9$.
 Then we   review Djokovic'  classification of 3-vectors in $\R^8$  and present  a classification of 5-forms  on $\R^8$ (Corollary \ref{thm:58}).  Djokovic's    classification  method   combines  some
 ideas   from Vinberg-Elashvili's  work  and   Galois  cohomology  method   for  classifying real  forms of a  complex orbit.
Note that  the   classification of 3-vectors   in $\R^8$ implies  the classification of 3-forms  in $\R^8$ (Proposition \ref{prop:isofv})  as well as  the classifications of
3-forms in  $\R^n$ for $n \le 7$ (Theorem \ref{thm:red}, Remark \ref{rem:lower}).
 Then we   review     a classification  of $\GL(8,\C)$-action on $\Lambda ^4 \C^8$  by Antonyan  \cite{Antonyan1981}, which is   important for classification    of 4-forms on $\R^8$.
At the end of Section 2  we review  a   scheme  of classification of 4-forms  on  $\R^8$  proposed by L\^e in 2011  \cite{Le2011}  and   Dietrich-Facin-de  Graaf's method  of classification of  3-forms  on $\R^8$ in  \cite{DFG2015}.

In Section  3, for $ k =2,3,4$,    we    compile  known results and  discuss  some open problems
on   necessary and sufficient   topological   conditions  for  the existence  of  a differential
$k$-form $\varphi$ of given  type $\St_{\GL(n, \R)} (\varphi(x)) $  on   manifolds   $M^n$  (in these  cases    the equivalence  class  of $\varphi(x)$ is defined uniquely  by the type of   the stablizer  of $\varphi(x)$, i.e., the   conjugation  class  of  $\St_{\GL(n,\R)}(\varphi(x))$  in  $\GL(n,\R)$).  In  dimension   $n = 8$  (and hence also for $n = 6,7$)
we   observe that   the  stabilizer  $\St_{\GL(n, \R)}(\varphi)$  of a  3-form $\varphi\in \Lambda^3\R^{n*}$   forms a
complete      system of invariants   of the action  of $\GL(n, \R)$ on $\R^n$ (Remark \ref{rem:auto8}).

We include two appendices in  this paper.   The first appendix  contains  a   result due to H\^ong V\^an L\^e concerning  the existence  of
3-form of type  $\tilde  G_2$ on a smooth 7-manifold, which has  been    posted  in arxiv in  2007 \cite{Le2007}.  The second appendix   outlines   the   Galois cohomology   method for classification of real forms of   a complex orbit. This  appendix  is     taken from  a  private note   by  Mikhail     Borovoi  with his kind permission.

Finally  we  would like to emphasize that our  paper is not a bibliographical survey. Some important papers may have been
missed if they are not directly related to the main lines of our narrative.  We  also   don't     mention in this survey the relations  of   geometry  defined  by differential forms to  physics   and instead refer  the reader to \cite{Joyce2007}, \cite{FOS2016}, \cite{FOS2017}, \cite{Vaizman1983}.


\section{Classification of   $\GL(n, \R)$-orbits   of $k$-forms on $\R^n$}\label{sec:class}



\subsection{General theorems}\label{subs:general}
We begin   the classification  of $\GL(n, \R)$-orbits   on $\Lambda^k\R^{n*}$  with
the following  observation  that the orbit  of the standard action of  $\GL(n,\R)$ on $\Lambda ^k \R^n$ can be identified  with   the   orbit  of the standard  action of $\GL(n,\R)$ on $\Lambda ^k \R^{n*}$ by using
an isomorphism $\mu \in   Hom(\R^n,\R^{n*}) = \R^{n*} \otimes \R^{n*} \supset S^2\R^{n*}$.
Note that  there are several papers  and books devoted to the classification
of $k$-vectors  on $\R^n$ \cite[Chapter VII]{Gurevich1964} \footnote{under ``polyvectors" Gurevich meant  both covariant  and contravariant  polyvectors},  \cite{Djokovic1983}, \cite{VE1978}.  Hence we have  the following  well-known  fact,  see  e.g.,   \cite{NR1994},
\begin{proposition}\label{prop:isofv}
	There exists  a bijection  between the $\GL(n,\mathbb R)$-orbits in $\Lambda^k\mathbb R^n$ and
	$\GL(n,\mathbb R)$-orbits in $\Lambda^k\mathbb R^{n*}$.
\end{proposition}

Next we  shall    compare    $\GL^+(n, \R)$-orbits on $\Lambda ^k\R^{n}$  with  $\GL^+(n, \R)$-orbits
 on $\Lambda^{n-k}\R^{n*}$.  We take a  volume  form $\Omega \in \Lambda ^n \R^{n*} \setminus \{ 0\}$ and define  the Poincar\'e isomorphism $P_\Om:\Lambda^k\mathbb R^n\rightarrow\Lambda^{n-k}\mathbb R^{n*}, \: \xi \mapsto i_\xi  \Om$.  Since
 $\GL^+(n,\R)$ is a  direct product   of its   center $Z(\GL^+(n, \R))=\R^+$ with  its semisimple
 subgroup $\SL(n, \R)$,  for any  $\lambda \in \R$ the  group $\GL^+(n, \R)$  admits  a $\lambda$-twisted      action  on $\Lambda ^k\R^{n*}$ defined  as follows: $g_{[\lambda]} (\varphi) : = (\det g)^\lambda \cdot g(\varphi)$  for $g \in \GL^+(n, \R)$, $\varphi \in \Lambda^k\R^{n*}$, where
 $g(\varphi)$ denotes the standard  action   of $g$ on $\varphi$.

 Denote  also by $\mu$    the   isomorphism $\Lambda ^k \R^n \to \Lambda^k \R^{n*}$ induced  from a
 scalar  product  $\mu$ on $\R^n$.

\begin{lemma}\label{lem:poincare} The composition  $P_\Om \circ \mu^{-1}: \Lambda  ^k \R^{n*} \to \Lambda  ^{n-k} \R^{n*}$ is a  $\GL^+(n, \R)$-equivariant  map   where
	$\GL^+(n,\R)$ acts  on  $\Lambda ^k \R^{n*}$ by the  standard action  and  on $\Lambda ^{n-k}\R^{n *}$ by   the $(-1)$-twisted action.
\end{lemma}
\begin{proof}   Let $\varphi  = \mu (X)\in \Lambda ^k \R^{n*}$  and  $g \in \GL^+(n, \R)$.  Then
$$ P_\Om \circ \mu^{-1} (g^* \varphi) = P_\Om (g^{-1}\circ \mu^{-1} (\varphi) )= i_{g^{-1} \mu^{-1} (\varphi)} \Om$$
$$ =  (\det g) ^{-1}\cdot   g(i_{\mu^{-1}(\varphi)}\Om) = g_{[-1]}(P_\Om \circ \mu ^{-1} (\varphi)),$$
which proves the first  assertion of Lemma \ref{lem:poincare}.
\end{proof}

\begin{proposition}\label{prop:connected} (1) There is a  1-1  correspondence  between $\GL^+(n, \R)$-orbits of  $k$-forms  on $\R^n$ and  $\GL^+(n, \R)$-orbits  of $(n-k)$-forms  on $\R^n$. This
correspondence  preserves the openness  of $\GL^+(n, \R)$-orbits (and hence the openness of $\GL(n,\R)$-orbits).
	
(2)  The $\GL(n, \R)$-orbit  of  $\varphi ^k \in \Lambda^k \R^{n*}$ has two  connected components
if and only if $\St_{\GL(n,\R)}(\varphi^k)\subset  \GL^+(n, \R)$. In other cases  the $\GL(n, \R)$-orbit  of
$\varphi^k$  is connected.

(3)  Assume  that   $\varphi^k\in \Lambda ^k\R^{n*}$ is   degenerate. Then
the $\GL(n, \R)$-orbit  of $\varphi^k$ is  connected.
\end{proposition}
\begin{proof}  1.  The first assertion of Proposition  \ref{prop:connected}  is a consequence of Lemma
	\ref{lem:poincare}. 
	
	2. The second assertion of Proposition \ref{prop:connected} follows  from the fact that $\GL(n, \R)$  has  two connected  components.
	
	3.   Assume that $\varphi$ is degenerate. Then    $W: = \ker  L_{\varphi}$ is non-empty.  Let $W^\perp$ be  any  complement   to $W$  in $\R^n$ i.e., $\R^n = W \oplus W^\perp$.
	 Then $\GL(W)\oplus Id_{W^\perp}$  is a subgroup    of $St(\varphi)$. Since   this subgroup
	 has non-trivial intersection with  $\GL^-(n, \R)$,   this implies
	the last assertion  of Proposition \ref{prop:connected}  follows  from the second one. This completes the  proof   of Proposition \ref{prop:connected}.
	\end{proof}

The following  theorem due to Vinberg-Elashvili reduces  a classification  of (degenerate) $k$-forms   of  rank $r$ in $\R^n$  to
a classification  of  $k$-forms  on $\R^r$.  (Vinberg-Elashvili  considered  only the case  $k=3$  and the $\SL(n, \C)$-action on $\Lambda^3\C^n$ but their argument  works  for any $k$   and for $\GL(n, \R)$-action on $\Lambda ^k  \R^{n*}$.)

\begin{theorem}\label{thm:red} (cf. \cite[\S  4.4]{VE1978}, \cite[Lemma 3.2]{Ryvkin2016})  There is a 1-1  correspondence between  $\GL(n, \R)$-orbits
	of  $k$-forms  of   rank  less or equal  to $r$ on $\R^n$ and  $\GL(r, \R)$-orbits  of
	$k$-forms  on $\R^r$.
\end{theorem}

\subsection{Classification of $2$-forms  and $(n-2)$-forms on $\R^n$}\label{subs:gl}

From Proposition \ref{prop:connected} we   obtain immediately the following known theorem  \cite{Dieudonne1955}, cf. \cite[Theorem  34.9]{Gurevich1964}.

\begin{theorem}\label{thm:k=2}  (1) The   rank  of  a  2-form $\varphi \in \Lambda^2\R^{n*}$  is  a  complete	invariant    of  the  standard $\GL(n, \R)$-action on $\Lambda ^2 \R^{n*}$.  Hence $\Lambda ^2 \R^{n*}$   decomposes into $[n/2]+1$ $\GL(n,\R)$-orbits.
	
(2) The $\GL(n, \R)$-orbit of a 2-form $\varphi \in \Lambda ^2 \R^{n*}$  has   two  connected components  if and only if $n=2k$ and $\varphi$ has maximal rank.

(3)  If $\varphi$ is of maximal rank, then the $\GL(n,\R)$-orbit of  $\varphi$  is open   and its   closure  contains the   $\GL(n, \R)$-orbit  of   any degenerate  2-form on $\R^n$.
\end{theorem}

The classification  of $(n-2)$-forms  on $\R^n$   has been  done by Martinet \cite{Martinet1970}.
Martinet  used the  inverse Poincar\'e isomorphism  $P_\Om^{-1}: \Lambda ^{n-2} \R^{n*} \to \Lambda ^2 \R^n$  to define  {\it the length of $\varphi \in \Lambda ^{n-2}\R^n$}, denoted by $l(\varphi)$, to be the  half  of the rank  of  the bi-vector $P_\Om ^{-1}(\varphi)$ \footnote{the rank of a $k$-vector is defined   similarly as the rank of a $k$-form.}.  By Proposition  \ref{prop:connected} and  Theorem \ref{thm:k=2}
the map $P_\Om^{-1}$  induces  an  isomorphism  between the  $\GL(n,\R)$-orbits of  degenerate  $(n-2)$-forms
$\varphi$ on $\R^n$ and
degenerate bivectors $P_\Om^{-1} (\varphi)$ on $\R^n$.

$\bullet$  If  $2l(\varphi)< n$ then $\varphi$ has the following canonical form
\begin{equation}\label{eq:martinet}
\varphi= \sum_{i=1}^{l(\varphi)} \alpha_1 \wedge \cdots \alpha_{2i-2}\wedge \alpha _{2i+1}\wedge \cdots \wedge \alpha_n.
\end{equation}
By Theorem \ref{thm:k=2} (2) the orbit $\GL(n,\R)\cdot P_\Om^{-1}( \varphi)$ is  connected, and hence
by Proposition \ref{prop:connected}  the  orbit $\GL(n,\R)\cdot \varphi$ is connected.

$\bullet$ If  $2l(\varphi) = n$,  and $l(\varphi)$  is  odd,  then using Lemma \ref{lem:poincare} and Theorem \ref{thm:k=2}(2)  we conclude that  the   set
of  $(n-2)$-forms of length $l$   consists  of  two open connected $\GL(n, \R)$-orbits   that  correspond to the   sign of $\lambda = \lambda_\Om(\varphi)$ where
$$P_\Om ^{-1}(\varphi )= e_1 \wedge e_2 + \cdots + e_{2k-1}\wedge  e_{2k},$$
$$\Om= \lambda \, \alpha _1 \wedge \cdots \wedge \alpha_n,$$
\begin{equation}
\label{eq:sign}
\varphi = \lambda \,\sum_{i=1}^{l(\varphi)} \alpha_1 \wedge \cdots \alpha_{2i-2}\wedge \alpha _{2i+1}\wedge \cdots \wedge \alpha_n  \text{  and } \lambda = \pm 1.
\end{equation}

$\bullet$ If  $2l(\varphi) = n$ and  $l(\varphi) $ is  even, using the same argument  as  in the previous case, we conclude that   the   set
of  (n-2)-forms of length $l$    consists  of  one  open
$\GL(n, \R)$-orbit, which has two connected  components.

To summarize   Martinet's  result, we assign  the sign  $s_\Om(\varphi)$  of a (n-2)-form $\varphi\in
\Lambda^{n-2}\R^n$  to be  the number  $\lambda_\Om(\varphi) ^{l(\varphi)}$ if  $2l(\varphi) =n$,
and to be  1, if $2l(\varphi) < n$.

\begin{theorem}\label{thm:k=n-2} (cf. \cite[\S 5]{Martinet1970})  (1)  The  length $l(\varphi)$  and the sign $s_\Om(\varphi)$  of  a  (n-2)-form $\varphi\in \Lambda^{n-2}\R^{n*}$   form a   complete  system  of    invariants of  the  standard $\GL(n, \R)$-action on $\Lambda ^{n-2} \R^{n*}$.
	
	(2) The $\GL(n, \R)$-orbit of a (n-2)-form $\varphi \in \Lambda ^2 \R^{n*}$  has   two  connected components  if and only if $n=2k$, $l(\varphi) = n/2$ and $l$ is even.
\end{theorem}

\subsection{Classification of $3$-forms  and $6$-forms  on $\R^9$}\label{subs:3form}

We observe that  the  vector  space $\Lambda ^k \R^{n*}$  is a real  form  of  the complex vector space $\Lambda ^k\C^{n*}$. Hence,   for any $\varphi\in \Lambda^k \R^{n*}$
the  orbit  $\GL(n, \R)\cdot \varphi$ lies in the   orbit  $\GL(n, \C)\cdot \varphi$. We  shall say that  $\GL(n, \R)\cdot \varphi$ is   a {\it real form  of the complex orbit}  $\GL(n, \C)\cdot \varphi$.  It is known that every  complex orbit has only  finitely many real forms \cite[Proposition 2.3]{BC1962}.   Thus, the problem of classifying  of the  $\GL(n, \R)$-orbits
in $\Lambda^k\R^n$  can be  reduced to  the problem of classifying   the real forms of the
$\GL(n,\C)$-orbits  on $\Lambda ^k \C^n$. The classification of  $\GL(n, \C)$-orbits  on $\Lambda^3\C^n$ is trivial, if  $n \le 5$, cf. Proposition \ref{prop:connected}.
For  $n =6$ it was  solved by  W. Reichel  \cite{Reichel1907};  for $n=7$ it was  solved by J. A. Schouten \cite{Schouten1931}; for $n =8$ it was solved by Gurevich in 1935, see  also \cite{Gurevich1964}; and for $n =9$ it was  solved by Vinberg-Elashvili \cite{VE1978}. In fact
Vinberg-Elashvili  classified $\SL(9, \C)$-orbits on  $\Lambda^3 \C^9$, which  are in 1-1 correspondence  with $\SL(9, \C)$-orbits in $\Lambda^3 \C^{9*}$ and $\SL(9, \C)$-orbits on $\Lambda ^6\C^{9*}$. Since  the center  of $\GL(9, \C)$  acts on $\Lambda ^3 \C^9\setminus \{ 0 \}$ with the kernel  $\Z_3$, it is not hard  to  obtain
a classification  of $\GL(9,\C)$-orbits on $\Lambda ^3 \C^9$,   and hence on $\Lambda ^3 \C^{9*}$  and on $\Lambda ^6\C^{9*}$ from the classification of  the $\SL(9,\C)$-orbits on $\Lambda^3\C^9$.

As we have remarked  before,  there are infinitely  many  $\GL(n,\C)$-orbits on $\Lambda^3\C^9$,  and to solve this   complicated   classification problem Vinberg-Elashvili  made an important observation that the  standard $\SL(9,\C)$-action on $\Lambda ^3\C^9$ is equivalent  to the  action of the  adjoint group  $G_0^\C$ (also called the $\theta$-group) of  the  $\Z_3$-graded  complex  simple  Lie algebra
\begin{equation}\label{eq:e93}
\e_8 =  \g_{-1}^\C  \oplus \g_0^\C \oplus  \g_1^\C
\end{equation}
where $\g_0^\C= \ssl (9, \C)$,  $\g_1^\C= \Lambda ^3 \C^3$, $\g_{-1}^\C= \Lambda ^3 \C^{9*}$
and $G_0^\C = \SL(9, \C)/\Z_3$  is the  connected  subgroup, corresponding to the    Lie subalgebra $\g_0^\C$, of the  simply connected Lie group $E^\C_8$ whose Lie algebra is $\e_8$.
\begin{remark}\label{rem:grad}  Let $\g^\C$ be a complex Lie algebra. Any  $\Z_m$-grading  	$\g^\C: = \oplus _{i \in \Z_m}\g_i^\C$ on $\g^\C$ defines
	an automorphism  $\sigma \in \Aut(\g^\C)$ of order $m$ by setting  $\sigma(x) := \epsilon ^i x$  where
	$\epsilon = \exp(2\sqrt{-1}\pi/m)$  and $x \in \g_i^\C$.  Conversely,  any   $\sigma \in \Aut(\g^\C)$  of  order $m$ defines  a $\Z_m$-grading  $\g^\C: = \oplus _{i \in \Z_m}\g_i^\C$   by  setting $\g_i^\C := \{ x\in \g^\C|\,\sigma(x) =  \epsilon ^i x\}$.
\end{remark}
In \cite[\S 2.2]{VE1978}  Vinberg  and Elashvili   considered the automorphism $\theta ^\C$ of order 3 on $\e_8$  associated to the $\Z_3$-gradation  in (\ref{eq:e83}) \footnote{Automorphisms  of finite order of semisimple Lie algebras  have been  classified   earlier independently by Wolf-Gray \cite{WG1968} and Kac \cite{Kac1969}.}. To describe  $\theta^\C$   we  recall the root system $\Sigma$  of $\e_8$:
$$\Sigma = \{ \eps_i -\eps_j, \pm (\eps_i + \eps_j+ \eps_k)\},\, (i,j,k \text { distinct}), \sum _{i = 1} ^9\eps_i  = 0\}.$$
\begin{remark}
	\label{rem:chevalley}
 Given  a  complex semisimple Lie algebra $\g^\C$ let us choose a
Cartan subalgebra $\h_0^\C$  of $\g ^\C$.  Let $\Sigma $ be  the root system of $\g^\C$. Denote by $\{  H_\alpha, E _\alpha |\, \alpha \in \Sigma\}$  the   Chevalley  system in $\g^\C$  i.e.,
$H_\alpha \in \h_0  ^\C$    and  $E_\alpha$ is the  root  vector    corresponding   to $\alpha$  such that for any  $H \in  \h_0 ^\C$ we have  $[H, E_\alpha]  = \alpha (H)E_\alpha$,
 $[H_\alpha, E_\alpha] = 2 E_\alpha$ and $[E_\alpha ,E_{-\alpha}] = H_\alpha$ \cite[\S 32.2]{Jacobson1979}.
Then 
\begin{equation}
\g^\C = \oplus _{\alpha \in \Sigma^+_s} \la H_\alpha \ra _\C\oplus _{\alpha\in \Sigma^+} \la E_\alpha\ra _\C \oplus _{\alpha\in \Sigma^+} \la E _{ - \alpha}\ra _\C
\label{eq:root1}
\end{equation}
where $\Sigma^+\subset \Sigma$ denote the     system of positive roots, and $\Sigma^+_s$ - the  subset  of simple roots.
\end{remark}

The automorphism $\theta ^\C$ of order 3 on $\e_8$  is   defined as follows
$$ \theta^\C  _{|\la H_{\alpha}, E_{\alpha}, \:  \alpha  = \eps_i - \eps_j\ra _\C}  = Id, $$
$$\theta^\C _{|\la E_{\alpha}, \alpha = (\eps_i + \eps _j + \eps _k)\ra _\C } = \exp  (i2\pi/3) \cdot Id, $$
$$\theta^\C _{|\la  E_{\alpha}, \alpha = -(\eps_i + \eps _j + \eps _k)\ra _\C } = \exp  (-i2\pi/3) \cdot Id. $$

\begin{remark}\label{rem:normalform}  Let $\{H_\alpha, E_\alpha| \, \alpha  \in \Sigma\}$ be the Chevalley
	system  of a  complex  semisimple Lie algebra $\g^\C$. Then  $\{H_\beta, E_\alpha| \, \alpha \in\Sigma, \, \beta \in \Sigma^+_s\}$  is  a basis of {\it the normal form} $\g$, also called {\it split real form},  of  $\g^\C$. 
	 The normal form  of  the    complex simple Lie algebra
	$\e_8$  is    denoted by $\e_{8(8)}$,  and the  normal  form of $\ssl(n, \C)$
	is the  real simple Lie algebra $\ssl (n,\R)$. Clearly  the  Lie subalgebra $\e_{8(8)}$  has  the induced  $\Z_3$-grading     from  the one on $\e_8$  defined in (\ref{eq:e93})  (note that $\e_{8(8)}$ is not invariant under  $\theta^\C$), i.e., we have
	$$\e_{8(8)}= \g_{-1} \oplus \g_0 \oplus \g_1$$
	where $\g_i = \e_{8(8)} \cap \g_i ^\C$ is a real  form of $\g_i^\C$ for $ i \in \{-1, 0, 1\}$.
	 Hence there is a 1-1  correspondence
	between $\SL(9,\R)$-orbits on $\Lambda^3 \R^{9*}$ and  the   adjoint action  of the
	subgroup $G_0$, corresponding to the Lie subalgebra  $\g_0$,  of the Lie group $G_0^\C$.
\end{remark}

Now let $\F$ be  the  field  $\R$   or $\C$. Based on (\ref{eq:e83}), Remark \ref{rem:normalform},  and following \cite[\S 1]{VE1978}, \cite[Lemma 2.5]{Le2011}, we shall call   a nonzero element
$ x \in \Lambda ^3 \F^9$   {\it semisimple},  if its orbit $\SL(9, \F)\cdot x$ is closed in
$\Lambda^3 \F^9$, and {\it nilpotent},  if the closure of  its orbit $\SL(9, \F) \cdot x$  contains
the zero 3-vector. Our notion  of semisimple  and nilpotent elements  agrees with the
notion of semisimple  and nilpotent elements in   semisimple Lie algebras \cite{VE1978}, \cite{Le2011}, see also \cite{Djokovic1983} for an  equivalent  definition of semisimple and nilpotent elements in  homogeneous   components of  graded  semisimple Lie algebras.

\begin{example}\label{ex:nvb}(\cite[\S 4.4]{VE1978})  Let $x \in \Lambda ^3 \F^9$ be a degenerate
	vector of   rank $r \le  8$, where $\F = \R$ or $\C$. (The  definition of the  rank of  a  $k$-vector can be defined in the same way as the definition of the rank of a $k$-form).  Then  for any $\lambda \in \R$  there  exists   an element	$g \in \SL(9, \F)$  such that  $g \cdot x = \lambda \cdot x$. Hence  the closure  of the  orbit  $\SL(9, \F)\cdot  x$ contains    $0 \in \Lambda ^3 \F^9$  and therefore $x$ is a nilpotent  element.
\end{example}

\begin{proposition}\label{prop:decom1} Every nonzero $3$-vector $x$  in $\Lambda ^3 \F^9$  can be uniquely
	written as  $x =  p + e$, where  $p$ is a  semisimple $3$-vector, $e$ - a nilpotent 3-vector, and
	$p \wedge e = 0$.
\end{proposition}
Proposition \ref{prop:decom1} has been   obtained  by Vinberg-Elashvili in \cite{VE1978} for  the case $\F= \C$.  To prove Proposition \ref{prop:decom1}  for $\F =\R$,    we    use
the  Jordan  decomposition  of a homogeneous element in a real $\Z_m$-graded Lie semisimple  algebra and   a version of the Jacobson-Morozov-Vinberg theorem  for  real  graded semisimple Lie algebras \cite[Theorem 2.1]{Le2011}.

Using Proposition \ref{prop:decom1}, Vinberg-Elashvili proposed   the following  scheme   for their classification of    3-vectors on $\C^9$. First they classified  semisimple  3-vectors  $p$. The   $\SL(9,\C)$-equivalence class  of   semisimple 3-vectors $p$ has dimension 4 -  the dimension of a maximal  subspace  consisting of commuting   semisimple  elements  in $\g_1$.  Then   the equivalence classes of semisimple elements $p$ are  divided  into  seven  types  according  to the type  of the stabilizer subgroup $\St(p)$ and the subspace $E(p): =\{ x\in \Lambda ^3 \C^9|\:  p \wedge x = 0\}$.  We assign  a  3-vector on $\F^9$ to the same  family as its semisimple part.  Then Vinberg-Elashvili described  all possible nilpotent  parts  for each  family of 3-vectors.  When the semisimple part is $p$, the latter are all the nilpotent  3-vectors $e$ of the space $E(p)$. The classification  is made  modulo the action of $\St_{\SL(9,\C)}(p)$.
Note  that   there   is only finite  number of nilpotent orbits
in $E(p)$ for any semisimple 3-vector $p$.  Therefore  the   dimension of the  orbit space
$\Lambda ^3 \C^9/ \SL(9,\C)$ is     4, which is  the dimension of  the space  of all semisimple  3-vectors. 

To classify  semisimple elements  $p \in  \Lambda ^3 \C^9$ and
 nilpotent  elements in $E(p)$ Vinberg-Elashvili     developed further  the general method  invented by Vinberg \cite{Vinberg1975a, Vinberg1975b, Vinberg1976, Vinberg1979}  for the study of the  orbits
of  the   adjoint action   of the $\theta$-group on  $\Z_m$-graded  semisimple complex Lie algebras.

  Vinberg's method has been  developed   by Antonyan  for classification of 4-forms in $\C^8$, which we shall describe  in  more detail  in Subsection 2.5,
    by L\^e  \cite{Le2011}  and Dietrich-Faccin-de Graaf \cite{DFG2015} for  real graded semisimple Lie algebras.
    L\^e  method   utilizes the existence  of  a  compatible  Cartan involution  of   a  real  graded semisimple  Lie algebra  \cite{Le2011},   which    has the same  role  as   the  existence  of the  Cartan involution in the theory  of  $\Z_2$-graded    real semisimple Lie algebras.  Dietrich-Faccin-de Graaf \cite{DFG2015}   extended Vinberg's  method further,      using   L\^e's   theorem on the  existence of a  compatible  Cartan involution that reserves the grading. 
   As a result,  we  have  partial  results  concerning the  orbit space  of the standard $\SL(9, \R)$-action
 on $\Lambda^3 \R^{9*}$  (as well  as   partial  results concerning the  orbit  space of the
 standard action  of $\SL(8,\R)$ on  $\Lambda^4 \R^{8*}$  we mentioned above).   
 By Proposition \ref{prop:decom1},   and following Vinberg-Elashvili scheme,
 the  classification  of the orbits  of $\SL(9, \R)$-action  on $\Lambda^3\R^9$    can be reduced
 to the classification  of   semisimple  elements  $p$ in $\Lambda^3 \R^9$, which is the same as   the classification of  real  forms  of  $\SL(9, \C)$-orbits   of  semisimple elements $p$ in $\Lambda ^3 \C^9$  (the classification  of the  $\SL(9, \C)$-orbits    has been  given in \cite{VE1978})
 and the classification of  nilpotent elements $e \in \Lambda ^3 \R^9$      such that  $e \wedge  p = 0$.  Note  that  $e$ is  a nilpotent   element  in the  semisimple  component  $Z(p)'$ of the zentralizer $Z(p)$  of the  semisimple element $p$.
 Thus the  latter problem is reduced to the classification
 of real forms of   complex nilpotent  orbits  in   $\Z(p)' _{\otimes \C}$,  and the classification of  the latter orbits has  been
  done in \cite{VE1978}.    L\^e's method \cite{Le2011}  and Dietrich-Faccin-de Graaf's method  of classification of nilpotent  orbits  of real graded Lie algebras \cite{DFG2015}   give  partial  information   on  the real forms of    these nilpotent  orbits.  We shall discuss
 a  similar scheme  of  classification  of 4-forms  on $\R^8$ in Subsection  2.5. 
 Currently  we  consider   the Galois  cohomology method  for     classification
  of 3-forms  on $\R^9$   promising  \cite{BGL2019},  and   therefore  we include  an appendix  outlining the Galois cohomology  method in this  paper.

\subsection{Classification of $3$-forms and $5$-forms   on $\R^8$}\label{subs:38}

The classification  of  $3$-vectors (and hence $3$-forms)  on $\R^8$
has  been given  by Djokovic  in \cite{Djokovic1983}. Similar     to  \cite{VE1978},  see (\ref{eq:e93}),
Djokovic   made an important  observation   that  for $\F = \R$  (resp. for   $\F = \C$)   the standard $\GL(8,\F)$-action on $\Lambda^3\F^8$ is  equivalent to the  action of the adjoint group $\Ad  G_0$ of the    $\Z$-graded   Lie algebra  $\g= \e_{8(8)}$ (resp.  $\g= \e_8$)  on the    homogeneous component
$\g_1$ of degree 1, where
\begin{equation}\label{eq:e83}
\g	= \g_{-3} \oplus \g_{-2} \oplus \g_{-1} \oplus \g_0  \oplus \g_1 \oplus  \g_2  \oplus \g_3.
\end{equation}
Here $\Ad G_0 = \GL(8,\F)/\Z_3$ \cite[Proposition 3.2]{Djokovic1983},    $\g_{-3} = \F^{8*}$,
$\g_{-2}  = \Lambda^2 \F^8$, $\g_{-1} = \Lambda ^3 \F^{8*}$,  $\g_0 = \gl(8, \F)$, $\g_1 = \Lambda^3 \F^8$, $\g_2  = \Lambda^2 \F^{8*}$, $\g_3 = \F^8$.

Since there is only finite  number of $\GL(n, \F)$-orbits  in $\g_1$, any element  in $\g_1$ is
nilpotent.  To study   nilpotent   elements in $\g_1 = \Lambda^3 \R^8$, as   Vinberg-Elashvili   did   for   complex  nilpotent  3-vectors   on $\Lambda^3 \C^9$,   Djokovic      used a    real version
of  Jacobson-Morozov-Vinberg's  theorem  that associates with  each nilpotent element $e \in \g_1$   a semisimple element $h(e) \in \g_0$  and  a nilpotent element  $f\in \g_{-1}$  that  satisfy the following condition \cite[Lemma 6.1]{Djokovic1983}
\begin{equation}\label{eq:jmv}
[h, e] = 2e, \: [h,f] = -2f, \: [e, f] = h.
\end{equation}
Element $h$ is defined  by $e$ uniquely  up to conjugation  and  $h = h(e)$ is called {\it a characteristic }  of $e$ \cite[Lemma 6.2]{Djokovic1983}, see also  \cite[Theorem 2.1]{Le2011}  for a     general statement.  Given $e$ and $h$, element  $f$ is defined uniquely.   A triple $(h, e,f)$  in (\ref{eq:jmv})  is called {\it  an $\ssl_2$-triple}, which we    shall denote by $\ssl_2(e)$.  With help  of  $\ssl_2(e)$-triples  Djokovic  classified     real forms  of nilpotent  orbits $\GL(8, \C)\cdot e$, where $e \in \g_1 = \Lambda ^3 \C^8$, as follows.  Denote by $Z_{\GL(8, \C)}(\ssl_2(e))$ the centralizer  of $\ssl_2(e)$ in $\GL(8, \C)$.  Let $\Phi = \Z_2$ be the  Galois group of the field extension of  $\C $ over $\R$. Then Djokovic proved that  there is a bijection from  the Galois cohomology $(\Phi, Z_{\GL(8,\C)}(\ssl_2(e )))$ to the   set of $\GL(8,\R)$-orbits contained in $\GL(8,\C) \cdot e$ \cite[Theorem 8.2]{Djokovic1983}.   A similar argument  has been   first used by  Revoy \cite{Revoy1979}  and  later by  Midoune  and Noui  for classification  of alternating forms   in dimension 8 over a  finite field \cite{MN2013}. Recall that     classification  of  $\GL(8,\C)$-orbits has been obtained by
Gurevich   and  later  this classification  is  also   re-obtained  by  Vinberg-Elashvili  in their classification of  3-vectors on $\C^9$. There are  altogether 23  $\GL(8,\C)$-orbits on $\Lambda^3 \C^8$.  In \cite{Djokovic1983} Djokovic gave  another  proof of this classification using  the   $\Z$-graded  Lie algebra $\e_8$ in  (\ref{eq:e83}). Finally Djokovic  computed  the  related  Galois cohomology to obtain the number   of real forms of  each complex  orbit  and  he also  found a canonical representation   of  each   $\GL(8,\R)$-orbit  on $\Lambda ^3 \R^8$.  The  space $\Lambda^3\R^8$  decomposes into 35 $\GL(8,\R)$-orbits.

\begin{remark}\label{rem:stable}  Since there is only finite  number  of   $\GL(8, \R)$-orbits
on $\Lambda ^3  \R^{8*}$,  there exists  $\varphi \in \Lambda ^3 \R^{8*}$  such that  the  orbit $\GL(8, \R)	\cdot \varphi$ is open in  $\Lambda ^3 \R^{8*}$.   Such a  $3$-form $\varphi$  is called
{\it stable}. Clearly  any  stable  $3$-form    $\varphi$ is nondegenerate,  i.e.,  $\rk \varphi = 8$.
In general, a $k$-form $\varphi$  on $\R^n$ is called {\it stable}, if the  orbit  $\GL(n, \R) \cdot \varphi$  is open in $\Lambda ^k \R^n$.  Clearly  any   symplectic  form is  stable.  It is not hard  to see that  if $\varphi \in \Lambda ^k \R^n$  is open,   and $k \ge 2$, then   either  $k =3$ and $n = 5, 6, 7,8$, or $k = 4$  and $n = 6, 7$, or $k = 5$ and $n =8$.  Stable forms  on $\R^8$   have been studied in deep  by  Hitchin \cite{Hitchin2001}, Witt \cite{Witt2005} and later by L\^e-Panak-Van\v zura  in \cite{LPV2008}, where they       classified     all stable  forms  on $\R^n$ (they  proved that   stable $k$-forms exist on $\R^n$  only in dimensions $n =  6, 7, 8$ if $ 3\le k \le n-k$), and  determined
their stabilizer   groups \cite[Theorem 4.1]{LPV2008}. 
\end{remark}

\begin{remark}\label{rem:lower}  Djokovic's   classification   of 3-vectors  on $\R^8$  contains
  the classification of 3-vectors  on  $\R^6$  and  the classification  of  3-vectors on $\R^7$  by Theorem \ref{thm:red}.
The classification  of 3-forms  on $\R^7$ has been first  obtained by Westwick \cite{Westwick1981}  by  adhoc method.  There  are    8   equivalence  classes  of   multisymplectic $3$-forms
on $\R^7$,  which  are  the real forms  of 5  equivalent classes of   multisymplectic  3-forms  on
$\C^7$, and there  are 6   equivalence  classes   of  $3$-forms  on $\R^6$, which are  the  real forms
of 5  equivalence  classes  of $3$-forms   on $\C^6$. The stabilizer  of $3$-forms  in $\R^6$  has been
determined  in \cite{Hitchin2000} and  the stabilizer  of multisymplectic  $3$-forms  in $\R^7$  has been defined  in \cite{BV2003}. The stabilizer  of   $3$-forms on $\F^7$   has been described by Cohen-Helminck in \cite[Theorem 2.1]{CH1988} for any  algebraically closed   field  $\F$.
	\end{remark}

\begin{remark}
	\label{rem:auto8}  There  are   21  equivalence  classes  of   multisymplectic  $3$-forms
	on $\R^8$ which are   the real forms  of  13  equivalence  classes  of multisymplectic  $3$-forms
	on $\C^8$  \cite[\S 9]{Djokovic1983}.  A complete  list  of  the stabilizer  groups
	$\St_{\GL(8,\R)}(\varphi)$ of each
	multi-symplectic  3-form $\varphi$  on $\R^8$ has not been  obtained  till now  according to our knowledge. The  stabilizer  $\St_{\GL(8,\C)}(\varphi)$  has been  obtained   by Midoune in his  PhD Thesis \cite{Midoune2009}, see also \cite{MN2013}.  In \cite{Djokovic1983}  Djokovic   computed  the dimension of
each   $\GL(8,\R)$-orbit in $\Lambda^3\R^8$  and the  centralizer  $Z_{\GL(8,\R)}(\ssl_2 (e))$  for
each nilpotent  element $e \in \e_{8(8)}$.  It follows  that the stabilizer  algebra  $Z_{\gl(8,\R)} (\varphi)$   of $3$-forms $\varphi \in \Lambda^3 \R^8$  forms a  complete  system  of invariants
of the  $\GL(8,\R)$-action on $\Lambda ^3 \R^8$.  In Proposition  \ref{prop:3multi} below  we show that  the stabilizer  of
any multisymplectic   $3$-form  $\varphi$ on $\R^8$  is not connected.
\end{remark}	
	
\begin{proposition}
	\label{prop:3multi}  For any    multisymplectic      $3$-form  $\varphi \in \Lambda^3\R^{8*}$   we  have
	$\St_{\GL(8,\R)}(\varphi) \cap \GL ^-  (8,\R) \not = \emptyset$. Hence the $\GL(8,\R)$-orbit of  any  $3$-form  on $\mathbb R^8$ is connected.
\end{proposition}

\begin{proof}  For each  equivalence  class  of a  $3$-form  $\varphi$  of rank 8  we choose  a canonical element $\varphi_0$ in the  Djokovic's    list \cite[p. 36-37]{Djokovic1983}. Then  we find
an element  $g \in \St_{\GL(8,\R)}(\varphi_0) \cap \GL^{-}(8,\R)$.   Hence   the  $\GL(n,\R)$-orbit  of each  multisymplectic  3-form on $\R^8$  is   connected. If  $\varphi$ is not    multisymplectic,  the  orbit  $\GL(8,\R)\cdot \varphi$ is connected  by Proposition \ref{prop:connected}. This completes  the proof  of  Proposition  \ref{prop:3multi}.
\end{proof}
\

Proposition  \ref{prop:3multi}  and  Proposition \ref{prop:connected}  imply immediately the following

\begin{corollary}\label{thm:58} (cf. \cite[Proposition 4.1]{Ryvkin2016}) The Poincar\'e  map $P_\Om$ induces  an  isomorphism  between
	$\GL(8,\R)$-orbits  on $\Lambda ^3 \R^8$  and $\GL(8,\R)$-orbit  on $\Lambda^5\R^{8*}$.  Each
	$\GL(8,\R)$-orbit  on $\Lambda ^5 \R^8$ is connected.
\end{corollary}


\subsection{Classification of 4-forms on $\R^8$}\label{subs:48}
Classification of    4-forms  on $\C^8$, whose equivalence  is  defined  via the standard action of $\SL(8,\C)$,   has  been given  by Antonyan \cite{Antonyan1981}, following the  scheme  proposed by Vinberg-Elashvili  for the classification of 3-vectors on $\C^9$.
In \cite{Le2011}  L\^e  proposed a  scheme    of classification of   4-forms  on $\R^8$
as  application of her  study  of the adjoint orbits  in $\Z_m$-graded  real semisimple Lie algebras. In this subsection  we outline Antonyan's  method   and L\^e's  method.

Let $\F = \C$  (resp. $\R$). Denote  by $\g$
the  exceptional   complex   simple Lie algebra $\e_7$ (rep.  $\e_{7(7)}$ - the split form  of $\e_7$). The starting   point  of Antonyan's       work on the classification  on 4-vectors on $\C^8$ (resp.   the starting  point  of L\^e's  scheme  of classification  of 4-forms  on $\R^8$)  is the  following observation,  cf. (\ref{eq:e93}), (\ref{eq:e83}).     The standard
$\GL(8, \F)$-action on  $\Lambda ^4 \F^8$  is equivalent  to  the action of  the $\theta$-group  of
the  $\Z_2$-graded   simple Lie algebra
\begin{equation}\label{eq:e84}
\g= \g_0 \oplus \g_1
\end{equation}
on its  homogeneous  component $\g_1$,  which  is  isomorphic to $\Lambda ^4 \F^8$.  Here $\g_0 = \ssl (8,\F)$.

Let us describe   the  components $\g_0$ and $\g_1$  in (\ref{eq:e84}) for the case $\F = \C$ using the  root decomposition of $\e_7$.  Recall that  $\e_7 $ has the following   root system:
	$$\Sigma=\{ \eps _i-\eps _j,  \eps_p + \eps _q + \eps _r + \eps _s, |\, i \not = j,  ( p,  q, r, s \text {  distinct}), \sum _{i =1} ^8 \eps _i = 0\}.$$
By Remark \ref{rem:grad},  the  $\Z_2$-grading  on $\e_7$ is defined  uniquely   by  an involution  $\theta ^\C$  of $\e_7$. In terms  of the Chevalley system  of  $\e_7$, see Remark \ref{rem:chevalley},   the  involution $\theta^\C$ is defined as follows:
	$${\theta^\C} _{| \h_0 } = Id,$$
	$${\theta^\C}  ( E _\alpha ) = E_\alpha, \text { if }  \alpha  = \eps _i - \eps _j ,$$
	$${\theta ^\C} ( E_\alpha ) =  - E _\alpha , \text { if } \alpha = \eps _i + \eps _j + \eps _k + \eps _l.$$
	
Note that  $\theta: = \theta ^\C_{ |\g= \e_{7(7)}}$  is  an involution  of $e_{7(7)}$ and it defines the induced  $\Z_2$-gradation  from   $\e_7$  on  $\e_{7(7)}$.

Following the Vinberg-Eliashivili scheme  of the classification  of 3-vectors on $\C^9$,
Antonyan  classified  $\SL(8,\C)$-equivalent 4-vectors on $\C^8$   by using the Jordan decomposition  (Proposition \ref{prop:decom1}).  First he classified  all  semisimple
4-vectors  on  $\C^8$   using  Vinberg's theory  on      finite  automorphisms of  semisimple algebraic groups \cite{Vinberg1975a}, which has been    employed  by Vinberg-Elashvili for the classification  of semisimple  3-vectors  as we mentioned above.
Next  we include    each semisimple  element   $x\in \g_1$  of the $\Z_2$-graded    complex  Lie algebra  $\e_7$ into  a  Cartan  subalgebra  of $\g_1$, which  is defined  as  a maximal subspace
in $\g_1$ consisting of  commuting  semisimple elements \cite{Vinberg1976} (this definition is also  applied to real  or complex   $\Z_m$-graded  semisimple Lie algebras $\g$).  If $\g$ is a    complex $\Z_m$-graded  sesmisimple Lie algebra, then  all the (complex) Cartan  subalgebras  in $\g_1$ are   conjugate under the action of  the   adjoint group $G_0^\C$. To reduce   the  classification of semisimple elements
in $\g_1$  further   we introduce the notion of the  Weyl group $W (\g, \Cc)$ of   a  complex  $\Z_m$-graded semisimple  Lie algebra  $\g$ w.r.t. to a Cartan subalgebra $\Cc \subset \g_1$ as  follows. Let
$G^\C$ be  the  connected  semisimple  Lie  algebra  having    the Lie algebra $\g$ and  $G_0 ^\C$  the Lie subgroup of the    $G^\C$  having  the Lie algebra $\g_0$.  We    define
$$N_0 (\Cc): =\{g\in G_0|\: \forall x\in \Cc\:  g(x) \in \Cc \} ,$$
$$Z_0 (\Cc): = \{ g\in G_0|\: \forall x \in \Cc\:  g(x) = x\}.$$
Then  $W(\g,\Cc): = N_0 (\Cc)/Z_0(\Cc)$.
The Weyl group  $W(\g, \Cc)$ is finite, moreover $W(\g, \Cc)$ is generated by complex reflections, which implies that  the algebra  of $W(\g,\Cc)$-invariants on $\Cc$  is free \cite{Vinberg1975a}. Furthermore,  two semisimple elements in $\Cc$ belong to the same
$G_0^\C$-orbit if and  only if  they are in the same orbit  of  the  $W(\g, \Cc)$-action on $\Cc$.
Antonyan  showed that  the Weyl group   $W(\e_7, \Cc)$  has order 2903040  and the generic  semisimple element has  trivial stabilizer.  He also    found a    basis   of   a   Cartan algebra  $\Cc \subset \g_1$, which is   also a  Cartan subalgebra  of  the Lie  algebra $\e_7$. Thus  the set of $\SL(8,\C)$-equivalent  semisimple   4-vectors on $\C^8$ has  dimension  7.   This set is divided  into 32  families   depending on the type of  the stabilizer of the action of  the  Weyl group  $W(\e_7,\Cc)$ on the  Cartan algebra $\Cc$.    For the classification of nilpotent   elements and  mixed   4-vectors  on $\C^8$  Antonyan  used  the Vinberg method of support  \cite{Vinberg1979}.

L\^e suggested the  following scheme of classification of the $\SL(8,\R)$-orbits   on $\Lambda^4\R^8$   \cite{Le2011}.  Observe that  we also  have  the Jordan decomposition of    each  element in  $\Lambda ^4\R^8$  into  a sum  of   a semisimple element  and a nilpotent element \cite[Theorem 2.1]{Le2011},  as in    Proposition \ref{prop:decom1}.  First, we  classify  semisimple  elements, using the   fact  that  every Cartan subspace $\Cc \subset \g_1$ is conjugated  to a standard Cartan subspace  $\Cc_0$  that   is invariant  under  the action  of a  Cartan involution $\tau_\u$ of   the  $\Z_2$-graded   Lie  algebra $\e_{7(7)}$ \cite{OM1980}.
The set of all standard  Cartan   subspaces   $\Cc_0\subset \g_1\subset \g= \e_{7(7)}$, and more  generally,   the set of all standard   Cartan subspaces  $\Cc \subset \g_1$ in any $\Z_2$-graded real semisimple Lie algebra  $\g$, has been  classified  by Matsuki and Oshima  in  \cite{OM1980}.  L\^e   decomposed
each  semisimple element  into  a  sum of  an elliptic  semisimple  element, i.e., a  semisimple  element whose adjoint action on $\g_{\otimes \C} = \e_7$ has purely imaginary  eigenvalues,  and a real semisimple element, i.e.,  a  semisimple  element whose adjoint action on $\g_{\otimes \C} = \e_7$ has  real  eigenvalues, cf. \cite{Rothschild1972}   for a similar  decomposition  of  semisimple elements  in a real sesimsimple Lie algebra.  The classification  of real semisimple elements and commuting  elliptic semisimple  elements in $\Cc_0 \subset \g_1$ is then reduced  to the classification  of  the orbits of  the
Weyl groups   of associated  $\Z_2$-graded  symmetric   Lie algebras  on their Cartan  subalgebras \cite[Corollary 5.3]{Le2011}.
    As  in  \cite{VE1978} and \cite{Antonyan1981},  the classification  of  mixed  4-vectors  on $\R^8$  is reduced  to   the classification of  their  semisimple parts    and the corresponding  nilpotent   parts.  The classification of nilpotent   parts    can be  done using algorithms   in real algebraic geometry
based  on  L\^e's theory  of   nilpotent orbits  in    graded semisimple Lie algebras \cite{Le2011}, that develops  further  Vinberg's method  of   support also called carrier algebra.
In \cite{DFG2015}   Dietrich-Faccin-de Graaf  developed   Vinberg's method  further  and   applied  their method
to classification of  the  orbits of homogeneous  nilpotent  elements in certain   graded real semisimple  Lie algebras. In particular,   they    have a new  proof  for  Djokovic's classification of  3-vectors  on $\R^8$.

\begin{remark}\label{rem:fin} (1) The method  of $\theta$-group   has been  extended  by Antonyan and Elashvili for classifications of spinors  in dimension 16 \cite{AE1982}.

(2)  Many results of classifications of $k$-vectors  over  the fields $\R$ and $\C$  have  their
analogues  over       other fields   $\F$  and  their  closures $\overline \F$  \cite{MN2013}. Over the field  $\F = \Z_2$  the classification of 3-vectors in $\F^n$  is related to some open problems in the theory of
self-dual codes  \cite{RS1998}, \cite{BV2017}. Till now there  is no  classification  of  3-vectors
in $\F^n$ if $n \ge 9$  except the cases  $\F  = \C$ (\cite{VE1978}) and $\F  = \Z_2$ (\cite{HP2019}).

\end{remark}




\section{Geometry defined by differential forms}\label{subs:subs}

In this section  we briefly  discuss  several  results  and open questions  on the existence 
of  differential $k$-forms   of given type on a smooth  manifold,  where $k =2,3,4$.

$\bullet$ Assume that  $k =2$   and  $\varphi$  is a closed 2-form  with constant  rank on $M^n$,  then $\varphi$  is called
{\it a  pre-symplectic form}  \cite{Vaizman1983}. Till now there is no general  necessary and sufficient  condition  for the existence of a  pre-symplectic  form $\varphi$  on  a manifold  $M^n$  except  the  case  that  $\varphi$ is a symplectic form.  Necessary conditions for the existence of  a symplectic form  $\varphi$  on $M^{2n}$  are the
existence of an almost complex structure on $M^{2n}$ and  if $M^{2n}$ is closed, the existence
of a cohomology class   $ a\in H^2 (M^{2n}; \R)$  with $a ^n >0$.   If $M^{2n}$ is open,  a theorem of Gromov \cite{Gromov1985, Gromov1986} asserts that the existence of an almost complex structure is also sufficient, his  argument   has been   generalized  in \cite{EM2002}  and used in  the   proof of Theorem \ref{thm:gt2}(2) below. Taubes using Seiberg-Witten theory   proved  that there  exist  a closed 4-manifold $M^4$ admitting   an almost complex structure  and $ a\in H^2 (M,\R)$ such that
$a^2 \not = 0$  but $M^4$ has no symplectic structure \cite{Taubes1995}.  Note that   for any    symplectic  form $\om$ on $M^{2n}$ there exists  uniquely  up to homotopy an almost complex structure   $J$ on $M^{2n}$ that  is compatible with $\om$, i.e., $g(X, Y ) : = \om (X, JY )$  is a Riemannian metric on $M^{2n}$. Connolly-L\^e-Ono  using the
Seiberg-Witten  theory  showed that a half  of all homotopy classes of almost complex structures on  a certain class of oriented compact 4-manifolds
is not compatible with any symplectic structure  \cite{CLO1997}.

$\bullet$  Manifolds  $M^{2n}$  endowed  with   a   nondegenerate   conformally  closed 2-form  $\om$, i.e., $d\om  = \theta \wedge  \om$ for   some  closed 1-form  $\theta$ on $M^{2n}$,   are called   {\it conformally   symplectic  manifolds}.  A necessary    condition   for the existence   of   nondegenerate   2-form $\om$ on $M^{2n}$ is the
  existence  of an almost complex structure  on $TM^{2n}$, which is equivalent to the existence  of  a section $J$ of the associated bundle $\SO(2n)/\U(n)$, see  \cite{Steenrod1951}  where a necessary  condition  for the existence  of a  section  $J$  has been determined in terms of the Whitney-Stiefel  characteristic  classes.  We don't  have   necessary   and sufficient  conditions  for  the existence  of   a general  conformally symplectic
  form  on  $M^{2n}$, except the  existence of  an almost  complex structure on $M^{2n}$.  In \cite{LV2015} L\^e-Van\v zura  proposed new cohomology theories
  of locally conformal symplectic  manifolds.

 $\bullet$ Assume  that $k = 3$  and  $\varphi$  is a stable  3-form on $M^8$.
 In \cite{NR1997}  Noui and Revoy   proved  that  the   Lie algebra  of the  stabilizer   of $\varphi$
 is    a real form  of   the Lie algebra $\ssl(3,\C)$.  Hence stable  3-forms  on $\R^8$  are equivalent  to the Cartan 3-forms on the real  forms  $\ssl(3, \R)$, $\su (1,2)$  and $\su(3)$ of the  complex Lie algebra $\ssl(3,\C)$. Later  in \cite{LPV2008}  L\^e-Panak-Van\v zura    reproved  the  Noui-Revoy result  by      associating to    each 3-form on  $\R^8$   various     bilinear forms,   which are  invariants of the $\GL(8,\R)$-action on $\Lambda^3 \R^{8*}$,   and studied  properties  of these  forms.  They computed the stabilizer group of a  stable form  $\varphi \in \Lambda ^3 \R^{8*}$ and found  a necessary and sufficient  condition for a closed orientable  manifold $M^8$
 to admit a stable 3-form  \cite[Proposition 7.1]{LPV2008}. In  \cite{Le2013} L\^e initiated  the study of geometry and topology of manifolds  admitting  a Cartan   3-form  associated with  a simple compact  Lie algebra.

 $\bullet$ Necessary and sufficient   conditions  for    a closed connected 7-manifold  $M^7$  to admit a
 multisymplectic 3-form  has been  determined in \cite{Salac2018}, see also Appendix \ref{sec:tg2}  below.  There  are two   equivalence  classes
 of stable 3-forms  on $\R^7$ with  the stablizer  groups $G_2$  and   $\tilde G_2$ respectively.
 Since  $G_2$  and   $\tilde G_2$  are exceptional  Riemannian and pseudo Riemannian  holonomy groups,
 manifolds  $M^7$ admitting   stable 3-form  of  $G_2$-type   (resp. of $\tilde G_2$-type) are in focus   of research  in Riemannian
 geometry (respectively in pseudo Riemannian geometry) \cite{Joyce2007}, \cite{LM2012}, \cite{KLS2018}.  As we have mentioned,  the study  of  geometries  of stable forms in dimension 6,7, 8  have been   initiated by Hitchin \cite{Hitchin2000, Hitchin2001}.

$\bullet$ It is  worth  noting  that    the algebra  of parallel  forms on  a quaternion  K\"ahler manifold  is generated by the quaternionic 4-form,  the algebra  of parallel forms  on a Spin(7)-manifold is generated by  the self-dual Cayley 4-form.   Riemannian  manifolds  admitting    parallel 2-forms of maximal  rank are  K\"ahler  manifolds, which  are the most studied   subjects in geometry, in particular  in  the  theory of minimal submanifolds, see e.g., \cite{LF1987}.

 \appendix
 \section{Manifolds admitting a $\tilde G_2$-structure}\label{sec:tg2}

 In  2000   Hitchin initiated  the study of   geometries   defined  by  differential forms  \cite{Hitchin2000},   and subsequently in \cite{Hitchin2001} he  initiated the  study of  geometries defined by stable  forms.   The latter geometries  have been   investigated further in   \cite{Witt2005},  \cite{LPV2008}.  A   necessary and  sufficient condition  for a manifold $M$ to admit  a stable form  $\varphi$ of $G_2$-type, i.e.,  the stabilizer of $\varphi$  is isomorphic  to the group $G_2$,    has been       found by Gray \cite{Gray1969}.   In this   Appendix we     state and prove a  necessary  and sufficient condition  for a manifold $M$ to admit  a stable form  $\varphi$ of $\tilde G_2$-type.  
  We recall that  a 3-form  $\varphi$ on
 $\R^7$ is called of $\tilde G_2$-type, if it lies on the $\GL(\R^7)$-orbit of a 3-form
 $$\varphi_0=\theta_1 \wedge \theta_2 \wedge \theta_3 +\alpha_1 \wedge \theta _1 + \alpha_2 \wedge \theta_2 + \alpha_3\wedge \theta _3.$$
 Here $\alpha_1, \alpha_2$ are 2-forms on $\R^7$ which can be written as
 $$\alpha_1 = y_1\wedge y_2 + y_3\wedge y_4, \: \alpha_2 = y_1 \wedge y_3 - y_2\wedge
 y_4, \: \alpha_3 = y_1 \wedge y_4 + y_2 \wedge y_3$$
 and $(\theta_1, \theta_2, \theta_3, y_1, y_2, y_3, y_4)$ is an  oriented basis of
 $\R^{7*}$.

 Bryant showed that    $\St_{\GL(7,\R)}(\varphi_0) = \tilde G_2$ \cite{Bryant1987}. He also  proved that   $\tilde G_2$ coincides with the automorphism
 group of  the split octonians \cite{Bryant1987}.

 \begin{theorem}\label{thm:gt2}  (1) Suppose that $M^7$ is a compact 7-manifold. Then  $M^7$ admits
 	a 3-form of $\tilde G_2$-type, if and only if $M^7$ is orientable  and spinnable.  Equivalently  the
 	first and second Stiefel-Whitney  classes of $M^7$ vanish.
 	
 	(2)  Suppose  that
 	$M^7$ is an open manifold which admits an embedding to a compact orientable  and spinnable
 	7-manifold. Then $M^7$ admits a closed 3-form  $\varphi$  of $\tilde G_2$-type.
 \end{theorem}

\begin{proof}
	First we recall that  the maximal compact Lie subgroup  of $\tilde G_2$  is $\SO(4)$.  This    follows from the Cartan theory on symmetric spaces. 
  We refer to \cite[p. 115]{HL1982} for an explicit
 	embedding of $\SO(4)$ into $G_2$. The reader can also check that the image of this  group  is also a subgroup
 	of $\tilde G_2\subset \GL (\R ^7)$.    We shall denote this image by
 	$SO(4)_{3,4}$.
 	
 	Now  assume that  a  smooth  manifold $M^7$ admits a $\tilde G_2$-structure. Then it must be orientable and spinnable, since  the maximal compact Lie subgroup $\SO(4)_{3,4}$ of  $G_2$ is also a  compact subgroup of the group $G_2$.
 	%

 	\begin{lemma}\label{lem:7cos}  Assume that $M^7$ is compact, orientable  and spinnable.  Then $M^7$ admits a
 		$\tilde G_2$-structure.
 	\end{lemma}
 	
 	\begin{proof}
 		Since $M^7$ is compact, orientable  and spinable, $M^7$ admits a SU(2)-structure \cite{FKMS1997}.
 		Since $SU(2)$ is a subgroup of $SO(4)_{3,4}$,  $M^7$  admits a $SO(4)_{3,4}$-structure. Hence $M^7$ admits a $\tilde  G_2$-structure.
 	\end{proof}

 	This completes the proof of  the first  assertion of Theorem \ref{thm:gt2}.

 	\

 	Let us  prove  the last statement of Theorem \ref{thm:gt2}.  Assume that $M^7$ is  a smooth  open manifold which admits an embedding into  a compact orientable  and spinnable  7-manifold. Taking  into account  the   first assertion  of  Theorem \ref{thm:gt2},  there exists a 3-form $\varphi$  on $M^7$ of $\tilde G_2$-type.  We shall use the  following theorem due to Eliashberg-Mishachev
 	to deform the 3-form
 	$\varphi$ to a closed 3-form $\bar \varphi$ of $\tilde G_2$-type on $M^7$.
 	
 	Let $M$ be a smooth manifold  and $a \in H^p(M, \R)$.  For  a subspace $\Rr \subset \Lambda ^p M$  we denote by
 	$Clo_a \Rr$ the  subspace of the space $\Gamma (M,\Rr)$ of smooth    sections $M \to \Rr$  that consists of closed $p$-forms
 	$\om \in \Gamma(M, \Rr) \subset \Om^p (M)$  such  that   $[\om] = a\in H^p(M,\R)$. Denote by $\Di (M)$  the diffeomorphism  group  of  $M$.

 	\begin{proposition}[Eliashberg-Mishashev Theorem] (\cite[10.2.1]{EM2002}) Let $M$ be an open manifold,
 		$a \in H^p (M,\R)$  and  $\Rr\subset \Lambda ^pM$ an open $\Di (M)$-invariant subset.
 		Then the inclusion
 		$$ Clo _a \Rr \INTO \Gamma(M, \Rr)$$
 		is a homotopy equivalence. In particular,
 		
 		- any $p$-form $\om \in \Gamma( M, \Rr)$ is homotopic  in  $\Rr$ to a closed form $\bar \om$;
 		
 		- any homotopy $\om_t \in \Gamma(M, \Rr)$   of $p$-forms which connects two closed forms
 		$\om_0, \om_1 $   such that $[\om_0] = [\om_1] = a \in H^p(M, \R)$ can be deformed in $\Rr$ into a homotopy of
 		closed forms $\bar \om_t$ connecting $\om_0$ and $\om_1$  such that $[\om_t] = a$  for all $t$.
 	\end{proposition}

 	Let $\Rr$  be the space
 	of all 3-forms of $\tilde G_2$-type on   $M^7$.  Clearly this space is an open $\Di (M^7)$-invariant  subset
 	of $\Lambda ^3 M^7$. Now we apply the  Eliashberg-Mishashev theorem to the  3-form $\varphi^3$
 	of $\tilde G_2$-type whose existence  has been proved above.  This  completes the proof   of  Theorem \ref{thm:gt2}.
 \end{proof}


\section{Classification of orbits over a nonclosed field\\ of characteristic 0}\label{sec:galois}
by Mikhail Borovoi

\

We consider a linear algebraic group $G$  with group of $k$-points $G(k)$ over an algebraically closed field $k$ of
characteristic 0. Assume that $G$ acts on a $k$-variety $X$ with set of $k$-points  $X(k)$,  and  assume that  we
know the classification of $G(k)$-orbits in $X(k)$,  e.g.,  $k = \C$, $G = \GL (9,\C)$,
$X = \Lambda^3 \C^9$.
Let $k_0$ be a subfield of $k$ such that $k$ is an algebraic closure of $k_0$. We
write $\Gamma= \Gal(k/k_0)$ for the Galois group  of the extension  $k$ over $k_0$.
If $ k_0 = \R$, then $\Gamma = \Gal(\C/\R) = \{ 1, \gamma\}$, where $\gamma$ is the complex conjugation.
Assume that we have   compatible $k_0$-forms   $G_0$ of $G$ and  $X_0$ of $X$.
We wish to classify $G_0(k_0)$-orbits in $X_0(k_0)$.
We start with one $G$-orbit $Y$ in $X$. We check whether $Y$ is $\Gamma$-stable. If not,
then $Y$ has no $k_0$-points.
Assume that $Y$ is $\Gamma$-stable. Then the $\Gamma$-action on $Y$ defines a $k_0$-model $Y_0$ of $Y$.
Now $G_0$ acts on $Y_0$ over $k_0$. We say
that $Y_0$ is (a twisted form of) a homogeneous space of $G_0$. We  ask

(1)  whether $Y_0$ has $k_0$-points;

(2) if  the answer to (1)  is positive, we wish to classify $G(k_0)$-orbits in $Y_0 (k_0 )$.

Question (1) is treated in  \cite{Borovoi1993}.
Assume that for our $Y$, the answer to question (1) is Yes. Let $y_0 \in Y_0 (k_0)$,
and let $H_0 = \St_{G_0} (y_0)$. Then we may write $Y_0 = G_0 /H_0$. The Galois group
$\Gamma = \Gal(k/k_0)$ acts compatibly on $G_0 (k) = G(k)$, $H_0 (k) = H(k)$, and $Y_0 (k) =
Y (k) = G(k)/H(k)$.

\begin{theorem}[{\cite{Serre1997}, Section I.5.4, Corollary 1 of Proposition 36}]
\label{thm:serre}
 There is a canonical bijection between the set of orbits $Y_0 (k_0 )/G_0 (k_0 )$ and the kernel
$\ker [H^1 (k_0 , H_0 ) \to  H^1 (k_0 , G_0 )]$.
\end{theorem}
Here $H^1 (k_0 , H_0 ) := H^1 (\Gamma, H_0 (k))$.


\section*{Acknowledgement}  The authors   would like to thank  Professor Alexander Elashvili and  Professor Andrea  Santi     for  their interest   in this  subjects   and      for   their suggestions of    references,  Professor Lemnouar Noui  for  sending  us  a  copy of the PhD Thesis  of Midoune \cite{Midoune2009}  and Professor Mahir Can    for   his  helpful comments    on a  preliminary version  of this   paper.
We  are grateful to  Professor Mikhail Borovoi  for  his help  in literature  and for his   writing up    an explanation of  the  Galois cohomology method   for finding  real forms  of complex orbits, which  we  put  as an Appendix   to this paper.

\end{document}